\documentclass[a4paper,11pt]{amsart}
\usepackage{amssymb}
\usepackage{amsmath,amscd}
\usepackage{mathtools}
\usepackage[paper=a4paper,left=3cm,right=3cm,top=3cm,bottom=2.3cm]{geometry}
\usepackage{enumerate}
\usepackage{url}
\usepackage{hyperref}
\usepackage{tikz}
\usepackage{tikz-cd}
\usetikzlibrary{arrows, matrix}

\newcommand{\CC}{\mathbb{C}}
\newcommand{\PP}{\mathbb{P}}
\newcommand{\ZZ}{\mathbb{Z}}
\newcommand{\QQ}{\mathbb{Q}}
\newcommand{\calO}{\mathcal{O}}
\newcommand{\calA}{\mathcal{A}}
\newcommand{\calM}{\mathcal{M}}
\newcommand{\calV}{\mathcal{V}}
\newcommand{\calH}{\mathcal{H}}
\newcommand{\Gm}{\mathbb{G}_m}

\newcommand{\et}{\mathrm{\acute{e}t}}
\newcommand{\sm}{\mathrm{sm}}
\newcommand{\sing}{\mathrm{sing}}
\newcommand{\dual}{\mathrm{dual}}
\newcommand{\stab}{\mathrm{s}}
\newcommand{\Gal}{\mathrm{Gal}}

\DeclareMathOperator{\M}{M}
\DeclareMathOperator{\id}{id}
\DeclareMathOperator{\Hom}{Hom}
\DeclareMathOperator{\End}{End}
\DeclareMathOperator{\Aut}{Aut}
\DeclareMathOperator{\Jac}{Jac}
\DeclareMathOperator{\Pic}{Pic}
\DeclareMathOperator{\NS}{NS}
\DeclareMathOperator{\Br}{Br}
\DeclareMathOperator{\ram}{ram}
\DeclareMathOperator{\alg}{alg}
\DeclareMathOperator{\pr}{pr}
\DeclareMathOperator{\Grass}{Grass}
\DeclareMathOperator{\Quot}{Quot}
\DeclareMathOperator{\PGL}{PGL}
\DeclareMathOperator{\GL}{GL}
\DeclareMathOperator{\Mat}{Mat}

\newtheorem{thm}{Theorem}[section]
\newtheorem{lemma}[thm]{Lemma}
\newtheorem{prop}[thm]{Proposition}
\newtheorem{cor}[thm]{Corollary}
\theoremstyle{definition}
\newtheorem{defi}[thm]{Definition}
\theoremstyle{remark}
\newtheorem{rem}[thm]{Remark}

\begin{document}

\title{A birational Torelli theorem with a Brauer class}

\author{Norbert Hoffmann}
\address{Department of Mathematics and Computer Studies, Mary Immaculate College, South Circular
  Road, Limerick, V94 VN26, Ireland}
\email{norbert.hoffmann@mic.ul.ie}

\author{Fabian Reede}
\address{Leibniz Universit\"at Hannover, Institut f\"ur algebraische Geometrie, Welfengarten 1, 30167 Hannover, Germany}
\email{reede@math.uni-hannover.de}


\begin{abstract}
  Let $\M_C( 2, \calO_C) \cong \PP^3$ denote the coarse moduli space of semistable vector bundles of rank $2$ with trivial
  determinant over a smooth projective curve $C$ of genus $2$ over $\CC$. Let $\beta_C$ denote the natural Brauer class
  over the stable locus. We prove that if $f^*( \beta_{C'}) = \beta_C$ for some birational map $f$ from $\M_C( 2, \calO_C)$
  to $\M_{C'}( 2, \calO_{C'})$, then the Jacobians of $C$ and of $C'$ are isomorphic as abelian varieties.
  If moreover these Jacobians do not admit real multiplication, then the curves $C$ and $C'$ are isomorphic.
  Similar statements hold for Kummer surfaces in $\PP^3$ and for quadratic line complexes.
\end{abstract}

\maketitle

\section{Introduction}

In this text, a (classical) \emph{Kummer surface} is a quartic surface $S \subset \PP^3$ over $\CC$ with
exactly $16$ nodes as singularities, the maximal possible number. Equivalently, $S$ is the quotient
of a principally polarized abelian surface $A$ over $\CC$ by its standard involution $-\id_A$.
These surfaces were first studied by Kummer in 1864, see \cite{kummer}.

We will work with Kummer surfaces $S$ obtained from quadratic line
complexes, and from moduli spaces of vector bundles of rank $2$ with trivial determinant over a curve of genus $2$.
In each of these situations, one has a natural Brauer class $\beta \in \Br( \PP^3 \setminus S)$ of order $2$,
which can be described as follows:
\begin{itemize}
 \item A \emph{quadratic line complex} $Q$ is a smooth intersection of the Grassmannian $\Grass_1( \PP^3)$
  of lines $\ell \subset \PP^3$ with another quadric in $\PP^5$. Its associated \emph{incidence correspondence}
  \begin{equation*}
  \begin{tikzcd}
    \PP^3 & I_Q \arrow{l}[swap]{\pr_p} \arrow{r}{\pr_\ell} & Q
  \end{tikzcd} 
  \end{equation*} 
  consists of all pairs $( p, \ell) \in \PP^3 \times Q$ such that the point $p$ is on the line $\ell$.
  The fibres of $\pr_p$ are conics in $\PP^2$, the locus where they are not smooth is a Kummer surface
  \begin{equation*}
    S_Q \subset \PP^3,
  \end{equation*}
  and the restricted bundle of smooth conics over $\PP^3 \setminus S_Q$ defines a Brauer class
  \begin{equation*}
    \beta_Q \in \Br( \PP^3 \setminus S_Q).
  \end{equation*}
  See Section \ref{sec:quadratic} for some references, and for more details.
 \item Let $C$ be a smooth projective curve of genus $g \geq 2$ over $\CC$. We denote by
  \begin{equation*}
    \M_C^{\stab}( r, L) \subseteq \M_C( r, L)
  \end{equation*}
  the open locus of stable vector bundles $E$ in the coarse moduli space of S-equivalence classes
  of semistable vector bundles $E$ of rank $r \geq 2$ with determinant $\Lambda^r E \cong L$ over $C$.
  The stable locus comes equipped with a natural Brauer class
  \begin{equation*}
    \beta_C \in \Br \M_C^{\stab}( r, L)
  \end{equation*}
  which can be interpreted as the obstruction against the existence of a Poincar\'{e} family (or universal family) 
  of vector bundles over $C \times \M_C^{\stab}( r, L)$, or equivalently also as the class of a corresponding
  moduli stack $\calM_C^{\stab}( r, L)$ as a gerbe with band $\Gm$ over $\M_C^{\stab}( r, L)$.
  See Section \ref{sec:moduli} for some references, and for more details.

  In the case of genus $g = 2$, rank $r = 2$ and trivial determinant $L = \calO_C$, one has
  \begin{equation*}
    \M_C( r, L) \cong \PP^3,
  \end{equation*}
  and the stable locus $\M_C^{\stab}( r, L)$ is the complement of a Kummer surface
  \begin{equation*}
    S_C \subset \PP^3,
  \end{equation*}
  which is the quotient of the Jacobian $A_C = \Jac( C)$ by its standard involution $-\id$.
  See Section \ref{sec:genus2} for some references, and for more details.
\end{itemize}

In each of these situations, the Brauer class $\beta$ can also be described solely in terms of the Kummer surface
$S \subset \PP^3$, namely as the unique class $\beta_S$ of order $2$ in $\Br( \PP^3 \setminus S)$
which is ramified along $S$ in such a way that the corresponding $2$-sheeted covering of $S$ is an abelian surface.
Restricting to the generic point of $\PP^3$, we can also view $\beta_S$ as an element of order $2$ in the Brauer group
of the rational function field $\CC( x_1, x_2, x_3)$.

The main observation of the present paper is that almost all of the geometric information can be recovered
solely from this Brauer class over the rational function field $\CC( x_1, x_2, x_3)$.
More precisely, we prove in particular the following:
\begin{thm}
  If $S, S' \subset \PP^3$ are Kummer surfaces with $f^*( \beta_{S'}) = \beta_S$ for some birational map
  \begin{equation*}
    f \colon \PP^3 \dashrightarrow \PP^3,
  \end{equation*}
  then $S \cong S'$ as algebraic surfaces.
\end{thm}
\begin{thm}
  If $Q, Q' \subset \Grass_1( \PP^3)$ are quadratic line complexes with $f^*( \beta_{Q'}) = \beta_Q$ for some birational map
  \begin{equation*}
    f \colon \PP^3 \dashrightarrow \PP^3,
  \end{equation*}
  then $S_Q \cong S_{Q'}$ as algebraic surfaces.
\end{thm}
\begin{thm}
  If $C, C'$ are curves of genus $2$ with $f^*( \beta_{C'}) = \beta_C$ for some birational map
  \begin{equation*}
    f \colon \M_C( 2, \calO_C) \dashrightarrow \M_{C'}( 2, \calO_{C'}),
  \end{equation*}
  then $\Jac( C) \cong \Jac( C')$ as abelian varieties.
\end{thm}
These statements are contained in Theorem \ref{thm:main}, Theorem \ref{thm:SQ_Cremona-iso} and
Theorem \ref{thm:SC_Cremona-iso} below, respectively. If the given datum is very general, in the sense that
the relevant abelian surface does not admit real multiplication, then we obtain the stronger conclusion
that the Kummer surfaces are projectively equivalent or that the curves are isomorphic;
see Corollary \ref{cor:A_no_real_mult} and Corollary \ref{cor:AC_no_real_mult}.
Without this extra assumption on the relevant abelian surface, there is at most finite ambiguity here;
see Corollary \ref{cor:finitely_many_S} and Corollary \ref{cor:finitely_many_C}.

\section{Kummer surfaces and an associated Brauer class}
We work over $\CC$. By a \emph{Kummer surface}, we mean a quartic surface
\begin{equation*}
  S \subset \PP^3
\end{equation*}
whose singular locus $S^{\sing} = S \setminus S^{\sm}$ consists of $16$ ordinary double points.

We will work with Kummer surfaces up to projective equivalence. 
Two hypersurfaces $S \subset \PP^n$ and $S' \subset \PP^n$ are called \emph{projectively equivalent}
if $f( S) = S'$ for some isomorphism
\begin{equation*}
  f \colon \PP^n \to \PP^n.
\end{equation*}

Each Kummer surface $S$ is known to be isomorphic to the quotient of an abelian surface $A$ by its standard
involution $-\id_A$ \cite[\S8]{hudson}. The resulting $2$-sheeted covering
\begin{equation*}
  \pi \colon A \twoheadrightarrow S
\end{equation*}
is branched over the $16$ points in $S^{\sing}$, whose inverse image consists of the $16$ points in the $2$-torsion
$A[2]$. The complex abelian surface $A$ is a quotient
\begin{equation*}
  A = V/\Lambda
\end{equation*}
of a complex vector space $V \cong \CC^2$ modulo a lattice $\Lambda \subset V$ with $\Lambda \cong \ZZ^4$.

The homology of Kummer surfaces is calculated in \cite{spanier}. For the convenience of the reader, 
we recall here what we will need in the sequel.
\begin{lemma}
  The fundamental group $\pi_1( S^{\sm})$ is a semidirect product $\Lambda \rtimes \ZZ/2$ 
  in which the nontrivial element of $\ZZ/2$ acts as $-\id$ on the normal subgroup $\Lambda$.
\end{lemma}
\begin{proof}
  Let $V^{\circ} \subset V$ and $A^{\circ} \subset A$ denote the inverse images of $S^{\sm} \subset S$.
  Their complements
  \begin{equation*}
    A \setminus A^{\circ} = A[2] \qquad\text{and}\qquad V \setminus V^{\circ} = \frac{1}{2} \Lambda
  \end{equation*}
  have real codimension $4$, so $\pi_1( A^{\circ}) = \Lambda$, and $V^{\circ}$ is simply connected. Therefore,
  \begin{equation*}
    V^{\circ} \twoheadrightarrow S^{\sm}
  \end{equation*}
  is a universal covering of $S^{\sm}$, so $\pi_1( S^{\sm})$ is the group of deck transformations
  $V^{\circ} \to V^{\circ}$. These consist of the translations $v \mapsto v + \lambda$
  and point reflections $v \mapsto \lambda - v$ with $\lambda \in \Lambda$.   
\end{proof}

\begin{cor}
  The singular homology $H_1( S^{\sm}, \ZZ)$ is isomorphic to $(\ZZ/2)^5$.
\end{cor}
\begin{proof}
  By the Hurewicz theorem, $H_1( S^{\sm}, \ZZ)$ is the maximal abelian quotient of $\pi_1( S^{\sm})$.
  The subgroup $2 \Lambda \subset \pi_1( S^{\sm})$ is normal, and the quotient $\Lambda/2\Lambda \rtimes \ZZ/2$
  is abelian because the induced action of $\ZZ/2$ on $\Lambda/2\Lambda$ is trivial. This abelian quotient
  \begin{equation*}
    \pi_1( S^{\sm}) \twoheadrightarrow \Lambda/2\Lambda \oplus \ZZ/2
  \end{equation*}
  is maximal, since each element $v \mapsto v + 2 \lambda$ of $2 \Lambda$ is a commutator in $\pi_1( S^{\sm})$.
\end{proof}

\begin{cor}
  The singular cohomology $H^1( S^{\sm}, \QQ/\ZZ)$ is isomorphic to $(\ZZ/2)^5$.
\end{cor}
\begin{proof}
  Since $S^{\sm}$ is arcwise connected, the universal coefficient theorem implies
  \begin{equation*}
    H^1( S^{\sm}, \QQ/\ZZ) \cong \Hom( H_1( S^{\sm}, \ZZ), \QQ/\ZZ ). \qedhere
  \end{equation*}
\end{proof}

\begin{cor}
  The \'{e}tale cohomology $H^1_{\et}( S^{\sm}, \QQ/\ZZ)$ is isomorphic to $(\ZZ/2)^5$.  
\end{cor}
\begin{proof}
  The comparison theorem \cite[Expos\'{e} XI, Th\'{e}or\`{e}me 4.4]{sga4} states that
  \begin{equation*}
    H^1_{\et}( S^{\sm}, \ZZ/n) \cong H^1( S^{\sm}, \ZZ/n). \qedhere
  \end{equation*}
\end{proof}

More precisely, the three theorems just quoted provide a \emph{canonical} isomorphism
\begin{equation} \label{eq:H1_to_pi_1}
  H^1_{\et}( S^{\sm}, \QQ/\ZZ) \longrightarrow \Hom( \pi_1( S^{\sm}), \QQ/\ZZ).
\end{equation}
There is a only one nontrivial homomorphism from $\pi_1( S^{\sm})$ to $\QQ/\ZZ$ that vanishes on all
elements of infinite order, namely the natural composition
\begin{equation} \label{eq:pi1_to_Z/2}
  \pi_1( S^{\sm}) = \Lambda \rtimes \ZZ/2 \twoheadrightarrow \ZZ/2 \hookrightarrow \QQ/\ZZ.
\end{equation}
\begin{defi} \label{def:alpha_S}
  Let $S \subset \PP^3$ be a Kummer surface. Then we denote by
  \begin{equation*}
    \alpha_S \in H^1_{\et}( S^{\sm}, \QQ/\ZZ)
  \end{equation*}
  the cohomology class corresponding, under the isomorphism \eqref{eq:H1_to_pi_1},
  to the composition \eqref{eq:pi1_to_Z/2}.
\end{defi}
By construction, this natural cohomology class $\alpha_S$ has order $2$, and corresponds to the \'{e}tale
restriction $A^{\circ} \twoheadrightarrow S^{\sm}$ of the $2$-sheeted covering $\pi \colon A \twoheadrightarrow S$.

\begin{rem}
  Every other nonzero class in $H^1_{\et}( S^{\sm}, \QQ/\ZZ)$ has order $2$ as well, and corresponds to an
  \'{e}tale restriction of a $2$-sheeted covering $S' \twoheadrightarrow S$ by a Kummer surface
  \begin{equation*}
    S' = A'/-\id_{A'}
  \end{equation*}
  where $A'$ is an abelian surface that admits an isogeny $A' \twoheadrightarrow A$ of degree $2$.

  This follows from the observation that the kernel of every other nontrivial homomorphism from
  $\pi_1( S^{\sm})$ to $\QQ/\ZZ$ intersects $\Lambda$ in a subgroup $\Lambda' \subset \Lambda$ of index $2$.
\end{rem}

Let $S' \subset \PP^3$ be another Kummer surface. Then $S' = A'/-\id_{A'}$ with $A'$ abelian.
\begin{lemma} \label{lem:S_determines_A}
  Each isomorphism $f \colon S \to S'$ lifts to an isomorphism of surfaces $g \colon A \to A'$.
\end{lemma}
\begin{proof}
  The uniqueness in the construction of the class $\alpha_S$ implies that 
  \begin{equation*}
    f^* \colon H^1_{\et}( {S'}^{\sm}, \QQ/\ZZ) \longrightarrow H^1_{\et}( S^{\sm}, \QQ/\ZZ)
  \end{equation*}
  maps $\alpha_{S'}$ to $\alpha_S$. Therefore, $f$ can be lifted to an isomorphism
  \begin{equation*}
    g^{\circ} \colon A^{\circ} \to {A'}^{\circ}.
  \end{equation*}
  This isomorphism automatically extends to an isomorphism $g$ from the integral closure $A$ of $S$
  in the function field $\CC( A^{\circ})$ to the integral closure $A'$ of $S'$ in $\CC( {A'}^{\circ})$.
\end{proof}
\begin{rem} \label{rem:S_determines_A}
  This lift $g \colon A \to A'$ of $f \colon S \to S'$ is in general not a group homomorphism,
  but the composition of a group homomorphism and the translation by a $2$-torsion point.
  In particular, $A$ and $A'$ are isomorphic as abelian surfaces.
\end{rem}

The projective embedding $S \hookrightarrow \PP^3$ determines a principal polarization $\theta \in \NS( A)$
in such a way that $2 \theta$ is the class of $\pi^* \calO_S( 1)$. Let $\theta' \in \NS( A')$ denote
the principal polarization determined in the same way by the projective embedding $S' \hookrightarrow \PP^3$.

\begin{cor}
  If the Kummer surfaces $S \subset \PP^3$ and $S' \subset \PP^3$ are projectively equivalent,
  then $(A, \theta)$ and $(A', \theta')$ are isomorphic as principally polarized abelian surfaces.
\end{cor}
\begin{proof}
  Let $f \colon \PP^3 \to \PP^3$ be an isomorphism with $f( S) = S'$.
  Due to Lemma \ref{lem:S_determines_A}, $f$ can be lifted to an isomorphism of surfaces $g \colon A \to A'$. Since
  \begin{equation*}
    f^* \calO_{S'}( 1) \cong \calO_S( 1)
  \end{equation*}
  as line bundles over $S$, we have $g^* \theta' = \theta$ in $\NS( A)$, and therefore
  \begin{equation*}
    (g \circ t_a)^* \theta' = \theta
  \end{equation*}
  in $\NS( A)$ for each translation $t_a \colon A \to A$ by a point $a \in A( \CC)$.

  By Remark \ref{rem:S_determines_A}, there is a $2$-torsion point $a \in A[2]$ such that $g \circ t_a$
  is an isomorphism of abelian surfaces, and hence an isomorphism of principally polarized abelian surfaces.
\end{proof}

\begin{lemma} \label{lem:Pic(S)_to_NS(A)_injective}
  The natural map $\pi^* \colon \Pic( S) \to \NS( A)$ is injective. 
\end{lemma}
\begin{proof}
  We use the commutative diagram
  \begin{equation*}
  \begin{tikzcd}
    \tilde{A} \arrow{r}{\tilde{\pi}} \arrow[swap]{d}{p_A} & \tilde{S} \arrow{d}{p_S}\\
    A \arrow{r}{\pi} & S
  \end{tikzcd}    
  \end{equation*}
  where $p_S$ is the blow-up of the $16$ points in $S^{\sing}$, $p_A$ is the blow-up of the $16$ points in $A[2]$,
  and $\tilde{\pi}$ is the induced map between these blow-ups. We claim that
  \begin{equation*}
    p_S^* \colon \Pic( S) \to \Pic( \tilde{S})
  \end{equation*}
  is injective. To check this, let $L$ be a line bundle over $S$ with $p_S^*( L) \cong \calO_{\tilde{S}}$. Since
  \begin{equation*}
    (p_S)_*( \calO_{\tilde{S}}) = \calO_S,
  \end{equation*}
  we then have $(p_S)_* p_S^*( L) \cong \calO_S$, and hence $L \cong \calO_S$ by the projection formula.

  Having proved that $p_S^*$ is injective, it now suffices to show that
  \begin{equation*}
    \tilde{\pi}^*: \Pic( \tilde{S}) \to \NS( \tilde{A})
  \end{equation*}
  is injective as well. For this, let now $L$ be a line bundle over $\tilde{S}$ such that
  \begin{equation*}
    \tilde{\pi}^*( L) \sim_{\alg} \calO_{\tilde{A}}
  \end{equation*}
  where $\sim_{\alg}$ denotes algebraic equivalence. Then
  \begin{equation*}
    \tilde{\pi}_* \tilde{\pi}^*( L) \sim_{\alg} \tilde{\pi}_*( \calO_{\tilde{A}})
  \end{equation*}
  where both are vector bundles of rank $2$ since $\tilde{\pi}$ is finite and flat of degree $2$.
  Using the projection formula and taking determinants, we conclude that
  \begin{equation*}
    L^{\otimes 2} \sim_{\alg} \calO_{\tilde{S}}.
  \end{equation*}
  Because $\tilde{S}$ is a (smooth) K3 surface, this implies $L \cong \calO_{\tilde{S}}$, as required.
\end{proof}

\begin{cor} \label{cor:S_projectively_equivalent_to_S'}
  If $(A, \theta)$ and $(A', \theta')$ are isomorphic as principally polarized abelian surfaces,
  then the Kummer surfaces $S \subset \PP^3$ and $S' \subset \PP^3$ are projectively equivalent.
\end{cor}
\begin{proof}
  Let $g \colon A \to A'$ be an isomorphism of principally polarized abelian surfaces.
  Then $g$ is in particular a group homomorphism, so
  \begin{equation*}
    (-\id_{A'}) \circ g = g \circ (-\id_A).
  \end{equation*}
  Therefore, $g$ descends to an isomorphism
  \begin{equation*}
    f \colon S \to S'.
  \end{equation*}
  Since moreover $g^*( \theta') = \theta$, the two line bundles $\calO_S( 1)$ and $f^* \calO_{S'}( 1)$ over $S$
  have the same image in $\NS( A)$. Using Lemma \ref{lem:Pic(S)_to_NS(A)_injective}, we conclude that
  \begin{equation*}
    f^* \calO_{S'}( 1) \cong \calO_S( 1)
  \end{equation*}
  as line bundles over $S$. Therefore, $f$ induces an isomorphism
  \begin{equation*}
    H^0( S', \calO_{S'}( 1)) \longrightarrow H^0( S, \calO_S( 1))
  \end{equation*}
  of $4$-dimensional vector spaces. Thus $f$ extends to an isomorphism $\PP^3 \to \PP^3$.
\end{proof}

\begin{prop} \label{prop:gysin_iso}
  The Gysin map $\partial \colon \Br( \PP^3 \setminus S) \longrightarrow H^1_{\et}( S^{\sm}, \QQ/\ZZ)$
  is an isomorphism.
\end{prop}
\begin{proof}
  Since $\PP^3 \setminus S$ is smooth and affine, the natural map
  \begin{equation*}
    \Br( \PP^3 \setminus S) \longrightarrow H^2_{\et}( \PP^3 \setminus S, \Gm)
  \end{equation*}
  is an isomorphism \cite{gabber,hoobler}. We use the Gysin sequence \cite[Corollary 16.2]{milneLEC} in the form
  given by \cite[Corollary 2.5]{bright}. Thus we obtain an exact sequence
  \begin{equation*}
    0 \longrightarrow \Br( \PP^3)[n] \longrightarrow \Br( \PP^3 \setminus S)[n]
      \longrightarrow H^1_{\et}( S^{\sm}, \ZZ/n) \longrightarrow H^3_{\et}( \PP^3, \mu_n).
  \end{equation*}
  Since $\Br( \PP^3) = 0$ and $H^3_{\et}( \PP^3, \mu_n) = 0$ \cite[Example 16.3]{milneLEC},
  it follows that the restrictions
  \begin{equation*}
    \partial_n \colon \Br( \PP^3 \setminus S)[n] \longrightarrow H^1_{\et}( S^{\sm}, \ZZ/n)
  \end{equation*}
  of the Gysin map $\partial$ are all isomorphisms. Hence $\partial$ is also an isomorphism.
\end{proof}

\begin{cor} \label{cor:Br}
  The Brauer group of $\PP^3 \setminus S$ is isomorphic to $(\ZZ/2)^5$.
\end{cor}

\begin{defi} \label{def:beta_S}
  Let $S \subset \PP^3$ be a Kummer surface. Then we denote by
  \begin{equation*}
    \beta_S \in \Br( \PP^3 \setminus S)
  \end{equation*}
  the class given by $\partial( \beta_S) = \alpha_S$ for the Gysin isomorphism $\partial$
  in Proposition \ref{prop:gysin_iso}.
\end{defi}
This class $\beta_S$ has by construction order $2$. It is preserved by projective equivalence, in the sense that
$f^*( \beta_{S'}) = \beta_S$ in $\Br( \PP^3 \setminus S)$ for each isomorphism $f \colon \PP^3 \to \PP^3$ with $f(S) = S'$.

\begin{rem}
  Let $X$ be a projective variety over $\CC$. Let $U \subseteq X$ be a smooth open subvariety.
  Then $\Br( U)$ embeds into $\Br \CC( X)$ according to \cite[Corollaire 1.10]{GBII}.

  Given another projective variety $X'$ over $\CC$, a smooth open subvariety $U' \subseteq X'$,
  Brauer classes $\beta \in \Br( U)$ and $\beta' \in \Br( U')$ and a birational map 
  \begin{equation*}
    f \colon X \dashrightarrow X',
  \end{equation*}
  we can thus compare $f^*( \beta')$ and $\beta$ in $\Br \CC( X)$. We say that \emph{the pairs $( X, \beta)$
  and $( X', \beta')$ are birational} if $f^*( \beta') = \beta$ in $\Br \CC( X)$ for some birational map $f$.
  We observe that such a birational map $f$ exists if and only if there is an isomorphism of $\CC$-algebras
  \begin{equation*}
    D_{\beta} \to D_{\beta'}
  \end{equation*}
  where $D_{\beta}$ and $D_{\beta'}$ are the unique central division algebras over the function fields
  $\CC( X)$ and $\CC( X')$ that have Brauer classes $\beta$ and $\beta'$, respectively.
\end{rem}

\begin{rem}
  Let $R$ be a discrete valuation ring with quotient field $R \subset K$ and residue field $R \twoheadrightarrow k$
  of characteristic zero. There is a natural short exact sequence
  \begin{equation*}
  \begin{tikzcd}
    0 \arrow{r} & \Br( R) \arrow{r} & \Br( K) \arrow{r}{\ram} & H^1_{\Gal}( k,\QQ/\ZZ) \arrow{r} & 0
  \end{tikzcd}
  \end{equation*}
  which describes the ramification of classes in $\Br( K)$ along the divisor $k$ and is therefore called
  the ramification sequence; see for example \cite[Chapter 10]{saltman}.

  To compare this ramification map to the Gysin map $\partial$, we consider the diagram
  \begin{equation*}
  \begin{tikzcd}
    \Br( \PP^3 \setminus S) \arrow{r}{\partial} \arrow[hookrightarrow]{d}
       & H^1_{\et}( S^{\sm}, \QQ/\ZZ) \arrow[hookrightarrow]{d}\\
    \Br \CC( \PP^3) \arrow{r}{\ram} & H^1_{\Gal} ( \CC( S), \QQ/\ZZ)
  \end{tikzcd}
  \end{equation*}
  in which both vertical maps are restrictions to the generic point.
  This diagram commutes up to sign; see for example \cite[Lemma 2.5]{BKT} or \cite[\S 3.3]{colliot} for more details.
  The sign does not matter here, because every class in $\Br( \PP^3 \setminus S)$ is $2$-torsion according to
  Corollary \ref{cor:Br}.

  In particular, the class $\beta_S \in \Br( \PP^3 \setminus S)$ is ramified along the irreducible divisor $S \subset
  \PP^3$, and its ramification there corresponds to the $2$-sheeted covering $\pi \colon A \twoheadrightarrow S$.
  These two properties characterize $\beta_S$ uniquely, according to Proposition \ref{prop:gysin_iso}.
\end{rem}

\begin{defi}
  Let $f \colon X \dashrightarrow X'$ be a birational map between smooth projective varieties over $\CC$.
  We say that an irreducible divisor $D \subset X$ is \emph{not exceptional} for $f$ if $f$ can be
  represented by an open immersion $f_U \colon U \to X'$ with $U \subseteq X$ open and $U \cap D \neq \emptyset$.  
\end{defi}
In this case, the closure of $f_U( D \cap U)$ in $X'$ does not depend on the choice of $U$, and is again
an irreducible divisor. We denote this divisor by $f( D) \subset X'$. Then $f$ restricts to a birational map
$f_{|D} \colon D \dashrightarrow f( D)$, and $f( D)$ is not exceptional for $f^{-1} \colon X' \dashrightarrow X$.
If $f( D)$ is also not exceptional for $g \colon X' \dashrightarrow X''$,
then $D$ is not exceptional for $g \circ f \colon X \dashrightarrow X''$.

\begin{defi}
  An algebraic variety is \emph{birationally ruled} if it is birational to a product of $\PP^1$ times
  another algebraic variety.
\end{defi}

\begin{prop} \label{prop:ruled}
  Let $f \colon X \dashrightarrow X'$ be a birational map between smooth projective varieties over $\CC$.
  Let $D \subset X$ be an irreducible divisor such that $D$ is not birationally ruled.
  Then $D$ is not exceptional for $f$.
\end{prop}
\begin{proof}
  The weak factorization theorem \cite[Theorem 0.1.1]{weak} states that $f$ is a composition of
  blow-ups of smooth centers and their inverses. Therefore, it suffices to treat the two special cases in
  which either $f$ or $f^{-1}$ is such a blow-up.

  Suppose first that $f \colon X \to X'$ is the blow-up of a smooth center $Z \subset X'$ of codimension $d \geq 2$.
  Then the inverse image $E := f^{-1}( Z) \subset X$ is a divisor that is birational to $\PP^{d-1} \times Z$,
  and hence birationally ruled. Every other divisor $D \subset X$ is not exceptional for $f$,
  since $f$ can be represented by the open immersion $X \setminus E \hookrightarrow X'$, 

  If $f^{-1} \colon X' \to X$ is the blow-up of a smooth center $Z \subset X$ of codimension $d \geq 2$,
  then every divisor $D \subset X$ is not exceptional for $f$, 
  since $f$ can be represented by the open immersion $X \setminus Z \hookrightarrow X'$.
\end{proof}

\begin{defi}
  A birational map
  \begin{equation*}
    f \colon \PP^n \dashrightarrow \PP^n
  \end{equation*}
  is a \emph{Cremona isomorphism} from a hypersurface $S \subset \PP^n$ to a hypersurface $S' \subset \PP^n$
  if $S$ is not exceptional for $f$, and $f( S) = S'$, and the restriction
  \begin{equation*}
    f_{|S} \colon S \dashrightarrow S'
  \end{equation*}
  of $f$ extends to an isomorphism $S \to S'$.
  If such a birational map $f$ exists, then the hypersurfaces $S$ and $S'$ are called \emph{Cremona isomorphic}.
\end{defi}

\begin{rem}
  Let $S \subset \PP^n$ and $S' \subset \PP^n$ be smooth hypersurfaces of degree $d$.
  In the case $n = 3$ and $d = 4$ of smooth quartic surfaces, Oguiso showed in \cite[Theorem 1.5]{oguiso}
  that $S$ and $S'$ can be Cremona isomorphic without being projectively equivalent.

  In all other case with $n \geq 3$ and $(n, d) \neq (3, 4)$, the smooth hypersurfaces $S$ and $S'$ can only be 
  Cremona isomorphic if they are projectively equivalent, by \cite[Theorem 1.8]{oguiso}.
\end{rem}

\begin{thm} \label{thm:main}
  Let $S \subset \PP^3$ and $S' \subset \PP^3$ be two Kummer surfaces, with associated abelian surfaces
  $A \twoheadrightarrow S$ and $A' \twoheadrightarrow S'$. Let
  \begin{equation*}
    f \colon \PP^3 \dashrightarrow \PP^3
  \end{equation*}
  be a birational map with $f^*( \beta_{S'}) = \beta_S$ in $\Br \CC( \PP^3)$.
  Then $f$ is a Cremona isomorphism from $S$ to $S'$.
  In particular, $S \cong S'$, and $A \cong A'$ as abelian surfaces.
\end{thm}
\begin{proof}
  The Kummer surface $S \subset \PP^3$ is not birationally ruled, since it has Kodaira dimension $0$, whereas   
  every birationally ruled surface has Kodaira dimension $-\infty$. Using Proposition \ref{prop:ruled},
  we conclude that $f$ can be represented by an open immersion
  \begin{equation*}
    f_U \colon U \to \PP^3
  \end{equation*}
  with $U \subseteq \PP^3$ open and $U \cap S \neq \emptyset$. Let $U' \subseteq \PP^3$ denote the image of $f_U$.

  Because the Brauer class $\beta_{S'}$ is ramified precisely in $S'$, its pullback $f_U^*( \beta_{S'})$
  is ramified precisely in the inverse image of $S'$. Since $f_U^*( \beta_{S'}) = \beta_S$ by
  assumption, this inverse image is $S \cap U$. Therefore, $f_U$ restricts to an isomorphism
  \begin{equation*}
    f_{U|S} \colon S \cap U \longrightarrow S' \cap U'.
  \end{equation*}
  Shrinking $U$ if necessary, we may assume that $S \cap U$ and $S' \cap U'$ are smooth.

  Due to the functoriality of the Gysin map $\partial$, our assumption $f_U^*( \beta_{S'}) = \beta_S$ implies that
  \begin{equation*}
    (f_{U|S})^* \colon H^1_{\et}( S' \cap U', \ZZ/2) \longrightarrow H^1( S \cap U, \ZZ/2)
  \end{equation*}
  maps $\alpha_{S'}$ to $\alpha_S$. Therefore, $f_{U|S}$ can be lifted to an isomorphism $g_U$ of open parts
  in the corresponding $2$-sheeted coverings $\pi: A \twoheadrightarrow S$ and $\pi': A' \twoheadrightarrow S'$:
  \begin{equation*} \begin{tikzcd}[column sep=large]
    A \arrow[swap,two heads]{d}{\pi}
      & \pi^{-1}( U) \arrow[two heads]{d} \arrow[left hook->]{l} \arrow{r}{g_U}
      & \pi'^{-1}( U') \arrow[two heads]{d} \arrow[hook]{r}
      & A' \arrow[two heads]{d}{\pi'}\\
    S
      & U \cap S \arrow[left hook->]{l} \arrow{r}{f_{U|S}}
      & U' \cap S' \arrow[hook]{r}
      & S'
  \end{tikzcd} \end{equation*}
  Because $A$ and $A'$ are abelian surfaces, the isomorphism $g_U$ between their open subvarieties
  $\pi^{-1}( U)$ and $\pi'^{-1}( U')$ extends to an isomorphism of surfaces
  \begin{equation*}
    g \colon A \to A'. 
  \end{equation*}

  Since the restriction $g_U$ of $g$ to $\pi^{-1}( U)$ descends to $f_{U|S}$ by construction,
  the two compositions $(-\id) \circ g$ and $g \circ (-\id)$ in the diagram
  \begin{equation*} \begin{tikzcd}
    A \arrow{r}{g} \arrow[swap]{d}{-\id} & A' \arrow{d}{-\id}\\
    A \arrow{r}{g} & A'
  \end{tikzcd} \end{equation*}
  agree on the dense open subvariety $\pi^{-1}( U)$ of their common source $A$. Therefore,
  \begin{equation*}
    (-\id) \circ g = g \circ (-\id)
  \end{equation*}
  on all of $A$. Consequently, $g$ descends to an isomorphism
  \begin{equation*}
    f_{|S} \colon S \to S'
  \end{equation*}
  which, by construction, extends the restriction $f_{U|S}$ of $f_U$.

  This proves that $f$ is a Cremona isomorphism from $S$ to $S'$, and in particular that $S \cong S'$.
  The remaining claim follows from Lemma \ref{lem:S_determines_A} together with Remark \ref{rem:S_determines_A}.
\end{proof}

\begin{cor} \label{cor:finitely_many_S}
  Fix a Kummer surface $S \subset \PP^3$. Up to projective equivalence, there are only finitely many
  Kummer surfaces $S' \subset \PP^3$ such that $( \PP^3, \beta_{S'})$ is birational to $( \PP^3, \beta_S)$.
\end{cor}
\begin{proof}
  Narasimhan and Nori proved that each abelian variety $A$ admits only finitely many principal polarizations
  up to automorphisms of $A$ \cite[Theorem 1.1]{naranori}. Using this, the claim follows from Theorem \ref{thm:main}
  in conjunction with Corollary \ref{cor:S_projectively_equivalent_to_S'}.
\end{proof}
\begin{cor}
  The action of the Cremona group $\Aut \CC( x_1, x_2, x_3)$ on the $2$-torsion in the Brauer group
  $\Br \CC( x_1, x_2, x_3)$ has uncountably many orbits.
\end{cor}
\begin{proof}
  Kummer surfaces depend on $18$ complex parameters, whereas $\dim \Aut( \PP^3) = 15$.
  Hence there are uncountably many Kummer surfaces up to projective equivalence.  
\end{proof}

Recall that an abelian surface $A$ over $\CC$ is said to \emph{admit real multiplication} if the ring
\begin{equation*}
  \End_{\QQ}( A) = \End( A) \otimes_{\ZZ} \QQ
\end{equation*}
contains a subring isomorphic to $\QQ( \sqrt{d})$ for some integer $d \geq 2$ which is not a square.
\begin{cor} \label{cor:A_no_real_mult}
  In the situation of Theorem \ref{thm:main}, suppose that the abelian surface $A \cong A'$ does not admit
  real multiplication. Then $S$ and $S'$ are projectively equivalent.
\end{cor}
\begin{proof}
  Combine Theorem \ref{thm:main} with Lemma \ref{lem:theta_unique} below
  and with Corollary \ref{cor:S_projectively_equivalent_to_S'}.
\end{proof}

\begin{lemma} \label{lem:theta_unique}
  Let $A$ be an abelian surface over $\CC$ with two different principal polarizations
  \begin{equation*}
    \theta \neq \theta' \in \NS( A).
  \end{equation*}
  Then $A$ admits real multiplication. 
\end{lemma}
\begin{proof}
  The two principal polarizations $\theta, \theta'$ define two isomorphisms
  \begin{equation*}
    \phi_{\theta}, \phi_{\theta'} \colon A \to \Pic^0( A)
  \end{equation*}
  as in \cite[\S 6, p. 60]{mumford}. These two isomorphisms differ by the automorphism
  \begin{equation*}
    \psi = \phi_{\theta}^{-1} \circ \phi_{\theta'} \colon A \to A.
  \end{equation*}

  Let $P_{\psi}( t)$ denote the characteristic polynomial of $\psi$ acting on $H_1( A, \ZZ) \cong \ZZ^4$.
  According to \cite[Subsection 2.2, p. 346]{debarre}, this polynomial $P_{\psi}( t)$ is the square of the polynomial 
  \begin{equation*}
    P( t) = \dfrac{( t \theta - \theta') \cdot ( t \theta - \theta')}{2} = t^2 - nt + 1
  \end{equation*}
  with $n = \theta \cdot \theta'$, using $\theta \cdot \theta = 2 = \theta' \cdot \theta'$. Hence in particular
  \begin{equation*}
    P( \psi)^2 = ( \psi^2 - n \psi + 1)^2 = 0
  \end{equation*}
  in $\End( A) \subset \End_{\QQ}( A)$.
  Since $P( t)$ has real roots by \cite[\S 16, p. 155]{mumford}, we have $n \geq 2$.

  If $A$ is isogenous to $E \times E$ for some elliptic curve $E$, then $\End_{\QQ}( A)$ contains a subring
  isomorphic to $\Mat_{2 \times 2}( \QQ)$, and therefore $A$ admits real multiplication.

  Suppose that $A$ is not isogenous to $E \times E$ for any elliptic curve $E$.
  Then $\End_{\QQ}( A)$ contains no nonzero nilpotent elements, according to Poincar\'{e}'s
  Complete Reducubility Theorem \cite[Corollary 5.3.8]{BL}. Therefore, we have already
  \begin{equation*}
    P( \psi) = \psi^2 - n \psi + 1 = 0.
  \end{equation*}
  For $n = 2$, this relation $(\psi - 1)^2 = 0$ would imply $\psi = \id$, which is absurd; hence $n \geq 3$.
  So the splitting field of $P( t)$ is the real quadratic number field $\QQ( \sqrt{n^2 - 4})$.
  As $\psi$ generates a $\QQ$-subalgebra of $\End_{\QQ}( A)$ isomorphic to this field, $A$ admits real multiplication.
\end{proof}

\begin{rem}
  Very general principally polarized abelian surfaces $A$ over $\CC$ do not admit real multiplication.
  More precisely, the abelian surfaces $A$ which admit real multiplication by a specific field $\QQ( \sqrt{d})$ form a
  surface $\calH_d$ in the $3$-dimensional moduli space $\calA_2$ of principally polarized abelian surfaces. 
  In particular, all principally polarized abelian surfaces $A$ outside these countably many surfaces
  $\calH_d \subset \calA_2$ have only one principal polarization.

  In the situation of Theorem \ref{thm:main}, suppose now that $\End_{\QQ}( A) \cong \QQ( \sqrt{d})$ and that $\Aut( A)$
  contains an element of negative norm. Then $A$ has only one principal polarization up to automorphisms of $A$,
  due to \cite[Proposition 4.1]{lange}. In this case, we can thus conclude as in Corollary \ref{cor:A_no_real_mult} above
  that $S$ and $S'$ are projectively equivalent.
\end{rem}

\section{Quadratic line complexes and an associated Brauer class} \label{sec:quadratic}
Let $\Grass_1( \PP^3)$ denote the Grassmannian of lines $\ell \subset \PP^3$ over $\CC$.
The Pl\"ucker embedding turns $\Grass_1( \PP^3)$ into a quadric hypersurface in a $\PP^5$.
A smooth divisor
\begin{equation*}
  Q \subset \Grass_1( \PP^3)
\end{equation*}
is called a \emph{quadratic line complex} if $Q$ is the intersection of $\Grass_1( \PP^3)$ with another
quadric hypersurface $q=0$ in the same $\PP^5$. It comes equipped with an \emph{incidence correspondence}
\begin{equation*}
\begin{tikzcd}
  \PP^3 & I_Q \arrow{l}[swap]{\pr_p} \arrow{r}{\pr_\ell} & Q
\end{tikzcd} 
\end{equation*} 
where $I_Q$ consists of all pairs $( p, \ell) \in \PP^3 \times Q$ such that the point $p$ is on the line $\ell$.

The projection $\pr_\ell$ is a $\PP^1$-bundle, as its fiber over a point $\ell \in Q$ is the line $\ell \subset \PP^3$.
The other projection $\pr_p$ is a conic bundle with degeneration, as its fiber
\begin{equation*}
  \pr_p^{-1}( p) \subset \PP^2
\end{equation*}
over a point $p \in \PP^3$ is the conic given by $q=0$ in the $\PP^2$ of lines $\ell \subset \PP^3$ containing $p$;
this fiber cannot be the whole $\PP^2$ by \cite[Section 6.2, Lemma on p. 762]{griffiths-harris}. The set
\begin{equation*}
  S_Q \subset \PP^3
\end{equation*}
of all points $p \in \PP^3$ with $\pr_p^{-1}( p)$ singular is a Kummer surface, and
\begin{equation*}
  \pr_p^{-1}(p) \text{ is } \begin{cases}
    \text{a smooth conic} & \text{if } p \in \PP^3 \setminus S_Q,\\
    \text{a union $\PP^1 \vee \PP^1$ of $2$ lines} & \text{if } p \in S_Q \setminus S_Q^{\sing},\\
    \text{a nonreduced double line $2\PP^1$} & \text{if } p \in S_Q^{\sing},
  \end{cases}
\end{equation*}
as proved for example in \cite[Section 6.2, pp. 762--771]{griffiths-harris}.
\begin{defi} \label{def:beta_Q}
  Let $Q \subset \Grass_1( \PP^3)$ be a quadratic line complex. Then we denote by
  \begin{equation*}
    \beta_Q \in \Br( \PP^3 \setminus S_Q)
  \end{equation*}
  the Brauer class of order $\leq 2$ given by the smooth conic bundle $\pr_p$ over $\PP^3 \setminus S_Q$. 
\end{defi}

Let $Q \subset \Grass_1( \PP^3) \subset \PP^5$ still be a quadratic line complex. The \emph{variety of lines}
\begin{equation*}
  A_Q \subset \Grass_1( \PP^5)
\end{equation*}
consists of all lines in $\PP^5$ that are contained in $Q$. Each such line is contained in a singular fiber
of $\pr_p$ according to \cite[Section 6.3, pp. 778--780]{griffiths-harris}, and the resulting map
\begin{equation*}
  \pi_Q \colon A_Q \longrightarrow S_Q
\end{equation*}
is a $2$-sheeted cover branched in the $16$ singular points of $S_Q$; moreover, $A_Q$ is an abelian surface,
and $S_Q$ is the quotient of $A_Q$ by the standard involution $-\id_{A_Q}$. Let
\begin{equation*}
  \alpha_Q  \in H^1_{\et}( S_Q^{\sm}, \QQ/\ZZ)
\end{equation*}
denote the class of order $2$ corresponding to the \'{e}tale restriction of $\pi_Q$ to the smooth locus
\begin{equation*}
  S_Q^{\sm} \subset S_Q.
\end{equation*}
By construction, $\alpha_Q$ equals the class $\alpha_{S_Q}$ associated with $S_Q$ in Definition \ref{def:alpha_S}. 
\begin{prop} \label{prop:beta_Q=beta_SQ}
  Let $Q \subset \Grass_1( \PP^3)$ be a quadratic line complex. Then
  \begin{equation*}
    \beta_Q = \beta_{S_Q} \in \Br( \PP^3 \setminus S_Q)
  \end{equation*}
  where $\beta_Q$ given by Definition \ref{def:beta_Q}, and $\beta_{S_Q}$ is given by Definition \ref{def:beta_S}.
\end{prop}
\begin{proof}
  Applying \cite[Theorem 1.1]{nitsure} to the conic bundle $\pr_p$, we get that the Gysin map
  \begin{equation*}
    \partial \colon \Br( \PP^3 \setminus S_Q) \longrightarrow H^1_{\et}( S_Q^{\sm}, \QQ/\ZZ)
  \end{equation*}
  sends $\beta_Q$ to $\alpha_Q$. Definition \ref{def:beta_S} states that $\partial$
  sends $\beta_{S_Q}$ to $\alpha_{S_Q} = \alpha_Q$ as well.
  Since $\partial$ is injective according to Proposition \ref{prop:gysin_iso}, this implies $\beta_Q = \beta_{S_Q}$.
\end{proof}

\begin{rem}
  In particular, the class $\beta_Q$ is non-zero. In other words, the smooth $\PP^1$-bundle $\pr_p$ over
  $\PP^3 \setminus S_Q$ is not Zariski-locally trivial. The non-vanishing of $\beta_Q$ has already been
  proved earlier in \cite[Proposition 8.1]{nararam3}, in \cite{newstead} and in \cite[Section 6]{newstead2}.
\end{rem}

\begin{thm} \label{thm:SQ_Cremona-iso}
  Let $Q, Q' \subset \Grass_1( \PP^3)$ be two quadratic line complexes. Let
  \begin{equation*}
    f \colon \PP^3 \dashrightarrow \PP^3
  \end{equation*}
  be a birational map with $f^*( \beta_{Q'}) = \beta_Q$ in $\Br \CC( \PP^3)$.
  Then $f$ is a Cremona isomorphism from $S_Q$ to $S_{Q'}$.
  In particular, $S_Q \cong S_{Q'}$, and $A_Q \cong A_{Q'}$ as abelian surfaces.
\end{thm}
\begin{proof}
  Due to Proposition \ref{prop:beta_Q=beta_SQ}, this follows from Theorem \ref{thm:main}.
\end{proof}

\section{Moduli Spaces of Vector Bundles} \label{sec:moduli}
Let $C$ be a smooth projective curve of genus $g \geq 2$ over $\CC$. For an integer $r \geq 2$ and
a line bundle $L \in \Pic( C)$ of degree $d \in \ZZ$, we denote by
\begin{equation*}
  \M_C( r, L)
\end{equation*}
the coarse moduli space of S-equivalence classes of semistable vector bundles $E$ over $C$ of rank $r$
with determinant $\Lambda^r E \cong L$. We denote by
\begin{equation*}
  \M_C^{\stab}( r, L) \subseteq \M_C( r, L)
\end{equation*}
the open locus of stable vector bundles. The following facts are known,
see for example \cite[Th\'eor\`eme 17]{seshadri} and \cite[Theorem 1]{nararam2}:
\begin{thm} \label{moduli1}
  The moduli space $\M_C( r, L)$ is a normal irreducible projective variety of dimension $(r^2-1)(g-1)$ over $\CC$.
  The open locus $\M_C^{\stab}( r, L)$ is smooth, and coincides with the smooth locus of $\M_C( r, L)$
  except when $g = r = 2$ and $d$ is even.
\end{thm}
The open locus of stable bundles comes equipped with a natural Brauer class
\begin{equation*}
  \beta_C \in \Br \M_C^{\stab}( r, L)
\end{equation*}
which can be described as follows, see \cite[Section 3]{KS} for more details:

Choosing an ample line bundle $\calO_C( 1)$ over $C$, and a sufficiently large integer $m$,
the coarse moduli space of stable vector bundles can be constructed as a GIT quotient
\begin{equation*}
  \M_C^{\stab}( r, L) = Q_{r, L}/\PGL_N
\end{equation*}
where $Q_{r, L}$ is a $\Quot$-scheme that parameterizes the vector bundles $E$ over $C$ in question
together with a basis of $H^0( C, E( m))$, and $N$ is the common dimension of the vector spaces $H^0( C, E( m))$,
so that $\GL_N$ acts on $Q_{r, L}$ by changing the chosen basis. Then the center of $\GL_N$ acts trivially,
the induced action of $\PGL_N$ is free, and the projection
\begin{equation*}
  Q_{r, L} \twoheadrightarrow \M_C^{\stab}( r, L)
\end{equation*}
is a principal $\PGL_N$-bundle. The group $\PGL_N$ acts on the space $\Mat_{N \times N}$ of quadratic matrices
by conjugation, and the associated fiber bundle
\begin{equation*}
  \Mat_{N \times N} \times^{\PGL_N} Q_{r, L} \twoheadrightarrow \M_C^{\stab}( r, L)
\end{equation*}
with fiber $\Mat_{N \times N}$ is an Azumaya algebra $\calA_{C, m}$. Its Brauer class $\beta_C := [\calA_{C, m}]$
depends neither on the choice of the ample line bundle $\calO_C( 1)$ nor on the integer $m$.

\begin{rem}
  Let $\calM_C( r, d)$ denote the moduli \emph{stack} of vector bundles $E$ of rank $r$ and degree $d$ over $C$.
  The condition $\det( E) \cong L$ defines a reduced closed substack
  \begin{equation*}
    \calM_C( r, L) \subset \calM_C( r, d),
  \end{equation*}
  which is an irreducible smooth Artin stack of dimension $(r^2-1)(g-1) - 1$ over $\CC$.
  The locus of stable vector bundles is a dense open substack
  \begin{equation*}
    \calM_C^{\stab}( r, L) \subset \calM_C( r, L),
  \end{equation*}
  which can be described in the above construction as the stack quotient $[Q_{r, L}/\GL_N]$. The universal property
  of the coarse moduli space $\M_C^{\stab}( r, L)$ provides a natural morphism
  \begin{equation*}
    \calM_C^{\stab}( r, L) \to \M_C^{\stab}( r, L),
  \end{equation*}
  which is a gerbe with band $\Gm$. The Brauer class $\beta_C$ is the class of this gerbe, see \cite{rationality}.
\end{rem}

\begin{rem}
  The Brauer class $\beta_C$ can also be interpreted as the obstruction against the existence of a Poincar\'{e}
  family (or universal family) of vector bundles over $C \times \M_C^{\stab}( r, L)$, see \cite[Section 3]{par}.
  Ramanan \cite[Theorem 2]{ramanan} proved that $\beta_C$ is non-zero if $\gcd( r, d) > 1$.

  More precisely, the order (also known as period or exponent) of the Brauer class $\beta_C$ is exactly $\gcd( r, d)$,
  see \cite[Proposition 5.1]{drezet-narasimhan} or \cite[Example 6.8]{par}. If $g \geq 3$ or $r \geq 3$, then
  $\beta_C$ generates the Brauer group of $\M_C^{\stab}( r, L)$, according to \cite[Theorem 1.8]{BBGN}.
\end{rem}

If the rank and the degree are coprime, that is $\gcd( r, d) = 1$, then the Brauer class $\beta_C$ vanishes.
Moreover, every semistable vector bundle is then stable, so we get:
\begin{thm}
  If $\gcd( r, d) = 1$, then $\M_C( r, L) = \M_C^{\stab}( r, L)$, and therefore the moduli space $\M_C(r,L)$
  is an irreducible smooth projective variety of dimension $(r^2-1)(g-1)$.
\end{thm}
The following Torelli type theorem holds in the case $\gcd( r, d) = 1$ of coprime rank and degree,
see \cite[Corollary]{mumnew}, \cite[Theorem 1]{tjurin} or \cite[Theorem 3]{nararam1}:
\begin{thm} \label{torelli}
  If $C$ and $C'$ are smooth projective curves of genus $g \geq 2$ with
  \begin{equation*}
    \M_C( r, L) \cong \M_{C'}( r, L')
  \end{equation*}
  for some line bundles $L \in \Pic( C)$ and $L' \in \Pic( C')$ of degree $d$ with $\gcd( r, d) = 1$, then
  \begin{equation*}
    C \cong C'.
  \end{equation*}
\end{thm}
Furthermore the birational type of the moduli space is known, see \cite[Theorem 1.2]{KS}:
\begin{thm}
  If $\gcd( r, d) = 1$ then the moduli space $\M_C( r, L)$ is a rational variety.
\end{thm}

We now focus our attention on the first case with $\gcd( r, d) > 1$, namely on curves $C$ of genus $g = 2$
and vector bundles $E$ of rank $r = 2$ with trivial determinant $\Lambda^r E \cong \calO_C$.

\section{The case of genus 2 and rank 2, trivial determinant} \label{sec:genus2}
Let $C$ be a smooth projective curve of genus $2$ over $\CC$. In this case, the moduli space $\M_C( 2, \calO_C)$
can be described explicitly as follows; see \cite{nararam2} for more details.

The \emph{theta divisor} $\Theta \subset \Pic^1( C)$ is the image of the canonical map $C \to \Pic^1( C)$ that
sends each point $x \in C$ to the class of the line bundle $\calO_C( x)$.
Riemann-Roch implies that the linear system $|n \Theta|$ has dimension $n^2 - 1$ for each $n \geq 1$.
\begin{thm}[Narasimhan-Ramanan]
  There is a canonical isomorphism
  \begin{equation*}
    \M_C( 2, \calO_C) \longrightarrow |2\Theta| \cong \PP^3
  \end{equation*}
  which sends each class $[E]$ of a vector bundle $E$ to the generalized theta divisor $\Theta_E$
  supported at all points $[L] \in \Pic^1( C)$ with $H^0( C, E \otimes L \neq 0)$. 
\end{thm}
\begin{proof}
  See \cite[Theorem 2]{nararam2}.
\end{proof}

So the moduli spaces $\M_C( 2, \calO_C)$ are all isomorphic to $\PP^3$, and hence smooth and rational.
Consequently, no Torelli type theorem in the form of Theorem \ref{torelli} holds for them.
This case also falls under the exception in Theorem \ref{moduli1}. Whereas the whole moduli space is smooth,
the open locus of stable vector bundles can be described as follows. 

Let $S_C$ denote the quotient of the Jacobian $A_C = \Pic^0( C)$ by the involution $-\id_{A_C}$.
The singular locus $S_C^{\sing} \subset S_C$ is the image of the $2$-torsion $A_C[2] \subset A_C$ and consists
of $16$ ordinary double points. One has a closed immersion
\begin{equation*}
  S_C \hookrightarrow \M_C( 2, \calO_C),
\end{equation*}
which is induced by the morphism $A_C \to \M_C( 2, \calO_C)$ that sends $[L]$ to $[L \oplus L^{-1}]$. According
to \cite[Proposition 6.3]{nararam2}, this turns $S_C$ into a Kummer surface in $\M_C( 2, \calO_C) \cong \PP^3$.
The open locus $\M_C^{\stab}( 2, \calO_C)$ of stable vector bundles is the complement of $S_C$ in $\M_C( 2, \calO_C)$.
\begin{rem}
  The Kummer surface $S_C \hookrightarrow \PP^3$ depends on the choice of an isomorphism $|2\Theta| \cong \PP^3$.
  We fix one such isomorphism in the sequel, and use it to identify $\M_C( 2, \calO_C)$ with $\PP^3$.
  A different choice would lead to a projectively equivalent Kummer surface in $\PP^3$.
\end{rem}

\begin{cor}
  The Brauer group of $\M_C^{\stab}( 2, \calO_C)$ is isomorphic to $(\ZZ/2)^5$.
\end{cor}
\begin{proof}
  Having just observed that the stable locus $\M_C^{\stab}( 2, \calO_C)$ is the open complement of
  a Kummer surface in $\M_C( 2, \calO_C) \cong \PP^3$, this now follows from Corollary \ref{cor:Br}.
\end{proof}

\begin{rem}
  In particular, the general description \cite[Theorem 1.8]{BBGN} of the Brauer group of $\M_C^{\stab}( r, L)$
  does not remain valid in the special case $g = r = 2$ and $d$ even.
\end{rem}

Since the curve $C$ has genus $2$, it is hyperelliptic, and its Weierstrass points are precisely the $6$ fixed
points of the hyperelliptic involution.
\begin{thm}[Newstead, Narasimhan-Ramanan] \label{thm:odd}
  If $x$ is a non-Weierstrass point on $C$, then there is a canonical isomorphism
  \begin{equation*}
    \M_C( 2, \calO_C( -x)) \longrightarrow Q_{C, x} \subset \Grass_1( \PP^3)
  \end{equation*}
  where $Q_{C, x}$ is a quadratic line complex in $\Grass_1( \PP^3)$.
  This isomorphism sends each class $[F]$ of a vector bundle $F$ to the line
  \begin{equation*}
    \ell_F := \{ [E] \in \M_C( 2, \calO_C) : \Hom( F, E) \neq 0\}.
  \end{equation*}
\end{thm}
\begin{proof}
  See \cite[Theorem 1]{newstead_odd} and \cite[Theorem 4]{nararam2}.
\end{proof}

The quadratic line complex $Q_{C, x}$ comes equipped with the incidence correspondence
\begin{equation*} \label{eq:hecke} \begin{tikzcd}
  \M_C( 2, \calO_C) \cong \PP^3 & I_{C, x} \arrow{l}[swap]{\pr_p} \arrow{r}{\pr_\ell}
    & Q_{C, x} \cong \M_C( 2, \calO_C( -x))
\end{tikzcd} \end{equation*} 
as recalled in Section \ref{sec:quadratic}. In terms of vector bundles, $I_{C, x}$ can be viewed  as a
\emph{Hecke correspondence}; see \cite[Section 4]{nararam1}. More precisely, $I_{C, x}$ parameterizes exact sequences
\begin{equation*}
  0 \longrightarrow F \longrightarrow E \longrightarrow \calO_x \longrightarrow 0
\end{equation*}
with $[F] \in \M_C( 2, \calO_C( -x))$ and $[E] \in \M_C( 2, \calO_C)$; here $\calO_x$ is the skyscraper sheaf at $x$.

As also recalled in Section \ref{sec:quadratic}, the quadratic line complex $Q_{C, x}$ defines a Kummer surface
\begin{equation*}
  S_{Q_{C, x}} \subset \M_C( 2, \calO_C) \cong \PP^3
\end{equation*}
and a $2$-sheeted covering $A_{Q_{C, x}} \twoheadrightarrow S_{Q_{C, x}}$ of it by an abelian surface,
namely by the variety
\begin{equation*}
  A_{Q_{C, x}} \subset \Grass_1( \PP^5)
\end{equation*}
of lines in $Q_{C, x}$. Actually, $S_{Q_{C, x}} = S_C$ due to \cite[Lemma 13.2]{nararam2},
and their $2$-sheeted coverings $A_{Q_{C, x}}$ and $A_C$ are canonically isomorphic by \cite[Theorem 5]{nararam2}.

In particular, the complement of $S_{Q_{C, x}}$ in $\M_C( 2, \calO_C)$ is again
the stable locus $\M_C^{\stab}( 2, \calO_C)$.

\begin{prop} \label{prop:beta_QC=beta_C}
  Let $C$ have genus $2$, and let $x$ be a non-Weierstrass point on $C$. Then
  \begin{equation*}
    \beta_{Q_{C, x}} = \beta_C \in \Br \M_C^{\stab}( 2, \calO_C)
  \end{equation*}
  where $\beta_{Q_{C, x}}$ is given by Definition \ref{def:beta_Q},
  and $\beta_C$ is the class defined in Section \ref{sec:moduli}.
\end{prop}
\begin{proof}
  Viewing the incidence correspondence $I_{C, x}$ as a Hecke correspondence in the above sense,
  we can identify the smooth conic bundle $\pr_p$ over $\M_C^{\stab}( 2, \calO_C)$ with the $\PP^1$-bundle
  whose fiber over a point $[E]$ is the projectivization $\PP E_x$ of the fiber $E_x$.

  Thus the Brauer class $\beta_{Q_{C, x}}$ of this smooth conic bundle $\pr_p$ is also the Brauer class of
  the Azumaya algebra $\calA_{C, x}$ over $\M_C^{\stab}( 2, \calO_C)$ whose fiber over a point $[E]$ is $\End( E_x)$.

  By construction, $\beta_C$ is the Brauer class of the Azumaya algebra $\calA_{C, m}$
  over $\M_C^{\stab}( 2, \calO_C)$ whose fiber over a point $[E]$ is $\End H^0( C, E( m))$.

  Because $\Aut( E) = \Gm$ acts with the same weight one on both $E_x$ and $H^0( C, E( m))$,
  there is a vector bundle $\calV$ over $\M_C^{\stab}( 2, \calO_C)$ whose fiber over a point $[E]$ is
  \begin{equation*}
    \calV_{[E]} := E_x^{\dual} \otimes H^0( C, E( m)).
  \end{equation*}
  By construction, $\calV$ is a bimodule that provides a Morita equivalence between the Azumaya algebras
  $\calA_{C, x}$ and $\calA_{C, m}$. Hence their Brauer classes $\beta_{Q_{C, x}}$ and $\beta_C$ are equal.
\end{proof}

\begin{thm} \label{thm:SC_Cremona-iso}
  Let $C$ and $C'$ be two smooth projective curves of genus $2$ over $\CC$. Let
  \begin{equation*}
    f \colon \M_C( 2, \calO_C) \dashrightarrow \M_{C'}( 2, \calO_{C'})
  \end{equation*}
  be a birational map $f^*( \beta_{C'}) = \beta_C$. Then $f$ is a Cremona isomorphism from $S_C$ to $S_{C'}$.
  In particular, $S_C \cong S_{C'}$, and the Jacobian $A_C$ is isomorphic to $A_{C'}$ as an abelian variety.
\end{thm}
\begin{proof}
  Due to Proposition \ref{prop:beta_QC=beta_C}, this follows from Theorem \ref{thm:SQ_Cremona-iso}.
\end{proof}

\begin{rem}
  Let $\theta_C$ denote the canonical principal polarization on the Jacobian $A_C$.
  Although $A_C \cong A_{C'}$ as abelian varieties, it is not clear at this point whether
  \begin{equation*}
    (A_C, \theta_C) \cong (A_{C'}, \theta_{C'})
  \end{equation*}
  as polarized abelian varieties. If the latter holds, then $C \cong C'$ according to
  the classical Torelli theorem, see for example \cite[Theorem 11.1.7]{BL}.
\end{rem}

\begin{cor} \label{cor:AC_no_real_mult}
  In the situation of Theorem \ref{thm:SC_Cremona-iso}, suppose moreover that the common Jacobian $A_C \cong A_{C'}$
  does not admit real multiplication. Then $C$ and $C'$ are isomorphic.
\end{cor}
\begin{proof}
  Combine Theorem \ref{thm:SC_Cremona-iso} with Lemma \ref{lem:theta_unique}
  and with the classical Torelli theorem.
\end{proof}

\begin{rem} \label{rem:stacks}
  Studying the birational type of the pair $( \M_C( 2, \calO_C), \beta_C)$ is equivalent to studying the birational
  type of the moduli stack $\calM_C( 2, \calO_C)$, because $\beta_C$ is the class of the $\Gm$-gerbe
  $\calM_C^{\stab}( 2, \calO_C) \twoheadrightarrow \M_C^{\stab}( 2, \calO_C)$.
  Indeed, a birational map
  \begin{equation} \label{eq:M-->M}
    f \colon \M_C( 2, \calO_C) \dashrightarrow \M_{C'}( 2, \calO_{C'})
  \end{equation}
  with $f^*( \beta_{C'}) = \beta_C$ is the same thing as a birational $1$-morphism
  \begin{equation} \label{calM-->calM}
    g \colon \calM_C( 2, \calO_C) \dashrightarrow \calM_{C'}( 2, \calO_{C'})
  \end{equation}
  of weight $1$. Conversely, each birational $1$-morphism $g$ as in \eqref{calM-->calM} has weight $1$ or weight $-1$
  by \cite[Lemma 3.8]{par}. Composing $g$ with the automorphism $E \mapsto E^{\dual}$ of $\calM_C( 2, \calO_C)$
  described in \cite[Example 3.7]{par} if necessary, we may assume without loss of generality that $g$ has weight $1$.
  Then $g$ induces a birational map $f$ as in \eqref{eq:M-->M} with $f^*( \beta_{C'}) = \beta_C$.
\end{rem}
 
\begin{cor} \label{cor:finitely_many_C}
  Fix a smooth projective curve $C$ of genus $2$ over $\CC$. Up to isomorphism, there are
  only finitely many smooth projective curves $C'$ of genus $2$ over $\CC$ such that
  the moduli stack $\calM_{C'}( 2, \calO_{C'})$ is birational to the moduli stack $\calM_C( 2, \calO_C)$.
\end{cor}
\begin{proof}
  Using \cite[Theorem 1.1]{naranori} as in the proof of Corollary \ref{cor:finitely_many_S},
  this follows from Theorem \ref{thm:SC_Cremona-iso} together with the classical Torelli theorem.
\end{proof}

\bibliographystyle{plain}
\bibliography{Artikel}

\end{document}